\documentclass[12pt,a4paper]{article}
\usepackage{mathrsfs}
\usepackage{indentfirst}
\usepackage{amsmath,amssymb,amsfonts,amsthm,graphics}
\usepackage{makeidx}
\usepackage{color}

\newtheorem{theorem}{Theorem}
\newtheorem{definition} {Definition}

\newtheorem{claim}{Claim}
\newtheorem{corollary}{Corollary}

\newtheorem{observation}{Observation}

\title{\bf New skew Laplacian energy of a simple
digraph\footnote{Supported by NSFC and the ``973'' program.}}
\author{\small Qingqiong Cai, Xueliang Li, Jiangli Song\\
{\small Center for Combinatorics and LPMC-TJKLC}\\ {\small Nankai
University}\\ {\small Tianjin 300071, China}\\ {\small Email:
cqqnjnu620@163.com, lxl@nankai.edu.cn, songjiangli@mail.nankai.edu.cn}}
\date{}

\begin{document}
\maketitle

\begin{abstract}
For a simple digraph $G$ of order $n$ with vertex set
$\{v_1,v_2,\ldots, v_n\}$, let $d_i^+$ and $d_i^-$ denote the
out-degree and in-degree of a vertex $v_i$ in $G$, respectively. Let
$D^+(G)=diag(d_1^+,d_2^+,\ldots,d_n^+)$ and
$D^-(G)=diag(d_1^-,d_2^-,\ldots,d_n^-)$. In this paper we introduce
$\widetilde{SL}(G)=\widetilde{D}(G)-S(G)$ to be a new kind of skew Laplacian
matrix of $G$, where $\widetilde{D}(G)=D^+(G)-D^-(G)$ and $S(G)$ is
the skew-adjacency matrix of $G$, and from which we define the skew
Laplacian energy $SLE(G)$ of $G$ as the sum of the norms of all the
eigenvalues of $\widetilde{SL}(G)$. Some lower and upper bounds of the new skew
Laplacian energy are derived and the digraphs attaining these bounds
are also determined.
\end{abstract}



\section {\large Introduction}
In chemistry, there is a close relation between the molecular
orbital energy levels of
 $\pi$-electrons in conjugated hydrocarbons and the eigenvalues of the corresponding molecular graphs.
 On these grounds, in $1970s$, Gutman \cite{gutman 1} introduced the concept of the
 energy for a simple undirected graph $G$:
\begin{center}
$E(G)=\sum\limits_{i=1}^{n}|\lambda_{i}|$
\end{center}
where $\lambda_{1},\lambda_{2},\ldots,\lambda_{n}$ are the
eigenvalues of the adjacency matrix $A(G)$ of $G$. Recall that $A(G)=[a_{ij}]$ is the $n\times n$ matrix,
where $a_{ij}= 1$ if $v_{i}$ and $v_{j}$ are adjacent, and $a_{ij}=
0$ otherwise.  Due to its
applications in chemistry, this concept has attracted much attention
and a series of related papers have been published. We refer to
\cite{gutman3,li} for details.

In spectral graph theory \cite{cvet}, the eigenvalues of several other matrices have been studied,
 of which the Laplacian matrix plays an important role. Therefore, based on the definition of graph energy,
 Gutman and Zhou \cite{gutman2} defined the Laplacian energy for a simple undirected graph $G$
possessing $n$ vertices and $m$ edges, which is given as follows:
\begin{center}
$LE_{g}(G)=\sum\limits_{i=1}^{n}|\mu_{i}-\frac{2m}{n}|$
\end{center}
where $\mu_{1},\mu_{2},\ldots,\mu_{n}$ are the eigenvalues of the
Laplacian matrix $L(G)=D(G)-A(G)$ of $G$. Recall that
$D(G)=diag(d_{1},d_{2},\ldots,d_{n})$ is the diagonal matrix of the
degrees of vertices. Some bounds were derived from the definition
in \cite{gutman2}.
\begin{theorem} \cite{gutman2} \label{thm1}
Let $G$ be a simple undirected graph possessing $n$ vertices, $m$
edges and $p$ components. Assume that $d_{1},d_{2},\ldots,d_{n}$ is
the degree sequence of $G$. Then

(i)~~$2\sqrt{M}\leq LE_{g}(G)\leq \sqrt{2Mn}$ ;

(ii)~~$LE_{g}(G)\leq \frac{2m}{n}p+\sqrt{(n-p)[2M-p(\frac{2m}{n})^{2}]}$;

(iii)~~If $G$ has no isolated vertices, then $LE_{g}(G)\leq 2M$,

\noindent where $M=m+\frac{1}{2}\sum\limits_{i=1}^{n}(d_{i}-\frac{2m}{n})^{2}$.
\end{theorem}
Moreover, Kragujevac \cite{kragu} considered another definition for
the Laplacian energy using the second spectral moment, namely
$LE_{k}(G)=\sum\limits_{i=1}^{n}\mu_{i}^{2}$. And the author proved
the following result.
\begin{theorem} \cite{kragu} \label{thm2}

(i)~~For any undirected graph $G$ on $n$ vertices whose degrees are
$d_{1},d_{2},\ldots,d_{n}$,
$LE_{k}(G)=\sum\limits_{i=1}^{n}d_{i}(d_{i}+1)$;

(ii)~~For any connected undirected graph $G$ on $n\geq2$ vertices, $LE_{k}(G)\geq 6n-8$,
where the equality holds if and only if $G$ is a path on $n$ vertices.

\end{theorem}

Since there are situations when chemists use digraphs rather than undirected graphs,
Adiga et al. \cite{adiga} first introduced the skew energy of a simple digraph.
Let $G$ be a simple digraph with vertex set $V(G)= \{v_{1},v_{2},\ldots,v_{n}\}$.
The skew-adjacency matrix of $G$ is the
$n\times n$ matrix $S(G)=[s_{ij}]$, where $s_{ij}= 1$ if $(v_{i},v_{j})$ is an arc of
 $G$, $s_{ij}= -1$ if $(v_{j},v_{i})$ is an arc of $G$, and $s_{ij}= 0$ otherwise. Then the skew energy of
 $G$ is the sum of the norms of all eigenvalues of $S(G)$, that is,
\begin{center}
$E_{s}(G)=\sum\limits_{i=1}^{n}|\lambda_{i}|$,
\end{center}
where $\lambda_{1},\lambda_{2},\ldots,\lambda_{n}$ are the
eigenvalues of the skew-adjacency matrix $S(G)$, which are all pure
imaginary numbers or $0$ since $S(G)$ is skew symmetric.

Similar to $LE_{k}(G)$, Adiga and Smitha \cite{adiga1} defined the
skew Laplacian energy for a simple digraph $G$ as
$$SLE_{k}(G)=\sum\limits_{i=1}^{n}\mu_{i}^{2},$$
where $\mu_{1},\mu_{2},\ldots,\mu_{n}$ are the eigenvalues of the
skew Laplacian matrix $SL(G)=D(G)-S(G)$ of $G$. In analogy with
Theorem \ref{thm1}, the following results were obtained.
\begin{theorem} \cite{adiga1} \label{thm3}

(i)~~For any simple digraph $G$ on $n$ vertices whose degrees are
$d_{1},d_{2},\ldots,d_{n}$, $SLE_{k}(G)=\sum\limits_{i=1}^{n}d_{i}(d_{i}-1)$;

(ii)~~For any connected simple digraph $G$ on $n\geq2$ vertices,
$2n-4\leq SLE_{k}(G)\leq n(n-1)(n-2)$, where the left equality holds
if and only if $G$ is the directed path on $n$ vertices and the
right equality holds if and only if $G$ is the complete digraph on
$n$ vertices.
\end{theorem}

Note that Theorem \ref{thm3} shows that the skew Laplacian energy of
a simple digraph defined in this way is independent of its
orientation, which does not reflect the adjacency of the digraph.
Being aware of this, later Adiga and Khoshbakht \cite{adiga2} gave
another definition
$SLE_{g}(G)=\sum\limits_{i=1}^{n}|\mu_{i}-\frac{2m}{n}|$, just like
$LE_{g}(G)$, and established some analogous bounds.
\begin{theorem} \cite{adiga2} \label{thm4}
Let $G$ be a simple digraph possessing $n$ vertices and $m$ edges.
Assume that $d_{1},d_{2},\ldots,d_{n}$ is the degree sequence of $G$
and $\mu_{1},\mu_{2},\ldots,\mu_{n}$ are the eigenvalues of the skew
Laplacian matrix $SL(G)=D(G)-S(G)$. Let
$\gamma_{i}=\mu_{i}-\frac{2m}{n}$ and $|\gamma_{1}|\leq |\gamma_{2}|
\leq \ldots \leq |\gamma_{n}|=k$. Then

(i)~~$2\sqrt{M}\leq SLE_{g}(G)\leq \sqrt{2M_{1}n}$ ;

(ii)~~ $SLE_{g}(G)\leq k+\sqrt{(n-1)(2M_{1}-k^{2})}$;

(iii)~~If $G$ has no isolated vertices, then $SLE_{g}(G)\leq
2M_{1}$,

\noindent where $M=-m+\frac{1}{2}\sum\limits_{i=1}^{n}(d_{i}-\frac{2m}{n})^{2}$
 and $M_{1}=M+2m=m+\frac{1}{2}\sum\limits_{i=1}^{n}(d_{i}-\frac{2m}{n})^{2}$.
\end{theorem}

In 2010, Kissani and Mizoguchi \cite{perera} introduced a different
Laplacian energy for directed graphs, in which only the out-degrees of vertices are
considered rather than both the out-degrees and in-degrees. Let $G$
be a digraph on n vertices. Suppose that
$\mu_{1},\mu_{2},\ldots,\mu_{n}$ are the eigenvalues of the matrix
$L^+(G)=D^+(G)-A^+(G)$, where
$D^+(G)=diag(d^+_{1},d^+_{2},\ldots,d^+_{n})$ is the diagonal matrix
of the out-degrees of vertices in $G$, and $A^+(G)=[a_{ij}]$ is the
$n\times n$ matrix, where $a_{ij}=1$ if $(v_{i},v_{j})$ is an
arc of $G$ and 0 otherwise. Then the Laplacian energy of $G$ defined
in \cite{perera} is
\begin{center}
$LE_{m}(G)=\sum\limits_{i=1}^{n}\mu_{i}^{2}$
\end{center}
By calculation, it is not hard to see that
$LE_{m}(G)=\sum\limits_{i=1}^{n}(d^+_{i})^{2}$ for a simple digraph $G$,
and $LE_{m}(G)=\sum\limits_{i=1}^{n}d^+_{i}(d^+_{i}+1)$ for a symmetric
digraph $G$. Furthermore, in \cite{perera} the authors found some
relations between undirected and directed graphs of $LE_{m}$ and used the
so-called minimization maximum out-degree (MMO) algorithm to
determine the digraphs with minimum Laplacian energy. The shortage
of this definition is that it does not make use of the in-adjacency
information of a digraph.

In this paper, we will introduce a brand-new definition for the skew
Laplacian energy of a simple digraph and obtain some lower and upper bounds
about it.

\section{\large A new skew Laplacian energy of digraphs}

We start with some notation and terminology which will be used in the sequel of this paper.

Given a simple digraph $G$ with vertex set
$V(G)=\{v_{1},v_{2},\ldots,v_{n}\}$, let $d_{i}^{+}$ and $d_{i}^{-}$ denote the out-degree and in-degree of a vertex $v_{i}$ in $G$ respectively. $A(G),S(G),D(G),D^{+}(G)$, $A^{+}(G)$ are defined as above.
Similar to $D^+(G)$, we define
$D^-(G)=diag(d^-_{1},d^-_{2},\ldots,d^-_{n})$, and
$\widetilde{D}(G)=D^+(G)-D^-(G)=diag(d^+_{1}-d^-_{1},d^+_{2}-d^-_{2},\ldots,d^+_{n}-d^-_{n})$.
Obviously, $D(G)=D^+(G)+D^-(G)$. Moreover, similar to $A^+(G)$, let $A^-(G)$ be the $n\times n$ matrix, where
$a_{ij}=1$ if $(v_{j},v_{i})$ is an arc of $G$ and 0 otherwise.
Clearly, $A^-(G)=(A^+(G))^T$. It is easy to see that the adjacency
matrix of the underlying undirected graph $G_{U}$ of $G$ satisfies
$A(G_{U})=A^+(G)+A^-(G)$ and the skew-adjacency matrix of $G$ satisfies
$S(G)=A^+(G)-A^-(G)$. Note that the Laplacian matrix of the
underlying undirected graph $G_{U}$ of $G$ can be written as

\begin{eqnarray*}
L(G_{U}) &=& D(G_{U})-A(G_{U})   \\
   &=& (D^+(G)+D^-(G))-(A^+(G)+A^-(G))\\
   &=& (D^+(G)-A^+(G))+(D^-(G)-A^-(G))
\end{eqnarray*}

Inspired by this, we define a new kind of skew Laplacian matrix
$\widetilde{SL}(G)$ of $G$ as
\begin{eqnarray*}
  \widetilde{SL}(G) &=& (D^+(G)-A^+(G))-(D^-(G)-A^-(G)) \\
   &=& (D^+(G)-D^-(G))-(A^+(G)-A^-(G))\\
   &=& \widetilde{D}(G)-S(G)
\end{eqnarray*}

Let $\mu_{1},\mu_{2},\ldots,\mu_{n}$ be the eigenvalues of the skew
Laplacian matrix $\widetilde{SL}(G)=\widetilde{D}(G)-S(G)$. Since $\widetilde{SL}(G)$ is not symmetric, it does not give real eigenvalues always. However, we
have the following two observations about the eigenvalues of
$\widetilde{SL}(G)$:
\begin{observation}\label{obser1}
$\sum\limits_{i=1}^{n}\mu_{i}=\sum\limits_{i=1}^{n}(d^+_{i}-d^-_{i})=0.$
\end{observation}
\begin{proof}
The relation is evident from
$\sum\limits_{i=1}^{n}\mu_{i}=trace(\widetilde{SL}(G)).$
\end{proof}
\begin{observation}\label{obser2}
$0$ is an eigenvalue of $\widetilde{SL}(G)$ with multiplicity at least $p$, the
number of components of $G$.
\end{observation}
\begin{proof}
Let $\sigma_{\widetilde{SL}}(G)$ denote the set of eigenvalues of the skew
Laplacian matrix $\widetilde{SL}(G)$. Assume that $C_{1},C_{2},\ldots,C_{p}$ are
all the components of $G$. Clearly,
$\sigma_{\widetilde{SL}}(G)=\bigcup\limits_{i=1}^{p}\sigma_{\widetilde{SL}}(C_{i})$. So it
suffices to prove that $0\in \sigma_{\widetilde{SL}}(C_{i})$ for $1\leq i \leq
n$. Now we restrict our attention to the induced subgraph $C_{i}$.
The sum of each row in $\widetilde{SL}(C_{i})$ is $0$, thus $0$ is an eigenvalue
of $\widetilde{SL}(C_{i})$ with eigenvector $[1,1,\ldots,1]^{T}$.
\end{proof}
Note that for the Laplacian matrix of an undirected graph, $0$ is
also an eigenvalue with
 eigenvector $[1,1,\ldots,1]^{T}$.

Now we give the formal definition for a new kind of skew Laplacian
energy.
\begin{definition}
Let $G$ be a simple digraph on $n$ vertices. Then the skew Laplacian
energy of $G$ is defined as
\begin{center}
$SLE(G)=\sum\limits_{i=1}^{n}|\mu_{i}|,$
\end{center}
where $\mu_{1},\mu_{2},\ldots,\mu_{n}$ are the eigenvalues of the
skew Laplacian matrix $\widetilde{SL}(G)=\widetilde{D}(G)-S(G)$ of $G$.
\end{definition}
We illustrate the concept with computing the skew Laplacian energy of two digraphs.

\noindent\textbf{Example 1 :} Let $P_{4}$ be a directed path on four vertices with the arc set
$\{(1,2)(2,3)(3,4)\}$. Then
\begin{center}
$\widetilde{SL}(P_{4})=\left(
                          \begin{array}{cccc}
                            1 & -1 & 0 & 0 \\
                            1 & 0 & -1 & 0 \\
                            0 & 1 & 0 & -1 \\
                            0 & 0 & 1 & -1 \\
                          \end{array}
                        \right)$.
\end{center}
The eigenvalues of $\widetilde{SL}(P_{4})$ are $i\sqrt{2}$, $-i\sqrt{2}$, 0, 0,
and hence the skew Laplacian energy of $P_{4}$ is $2\sqrt{2}$.

\noindent\textbf{Example 2 :} Let $C_{4}$ be a directed cycle
on four vertices with the arc set $\{(1,2)(2,3)(3,4)(4,1)\}$. Then
\begin{center}
$\widetilde{SL}(C_{4})=\left(
                          \begin{array}{cccc}
                            0 & -1 & 0 & 1 \\
                            1 & 0 & -1 & 0 \\
                            0 & 1 & 0 & -1 \\
                            -1 & 0 & 1 & 0 \\
                          \end{array}
                        \right).$
\end{center}
The eigenvalues of $\widetilde{SL}(C_{4})$ are $2i$, $-2i$, 0, 0, and hence the
skew Laplacian energy of $C_{4}$ is $4$, which is the same as the
skew energy of $C_{4}$. Actually, we have the following more general
result:
\begin{theorem}
If $G$ is an Eulerian digraph, then $SLE(G)=E_{s}(G)$.
\end{theorem}
\begin{proof} Since $G$ is Eulerian, the out-degree and the in-degree are equal for each vertex in
$G$, and so $\widetilde{D}=0$, which results in $\widetilde{SL}(G)=-S(G)$, and
consequently $SLE(G)=E_{s}(G)$.
\end{proof}

\section{\large Some lower and upper bounds for the new $SLE(G)$}

This section is devoted to obtaining some lower and upper bounds for the
skew Laplacian energy $SLE(G)$ and determining the digraphs attaining these bounds.
\begin{theorem}\label{thm5}
Let $G$ be a simple digraph possessing $n$ vertices, $m$ edges and
$p$ compenents. Assume that $d^+_{i}$ ($d^-_{i}$) is the out-degree
(in-degree) of a vertex $v_{i}$ in $G$. Then

\begin{center}
 $2\sqrt{|M|}\leq SLE(G)\leq \sqrt{2M_{1}(n-p)},$
\end{center}
\noindent where
$M=-m+\frac{1}{2}\sum\limits_{i=1}^{n}(d^+_{i}-d^-_{i})^{2}$ and
$M_{1}=M+2m=m+\frac{1}{2}\sum\limits_{i=1}^{n}(d^+_{i}-d^-_{i})^{2}$.
Moreover, these bounds are sharp.
\end{theorem}
\begin{proof}
The proof is very similar to those of Theorems 1 and 4, but the
extremal digraphs are determined very differently. Let
$\mu_{1},\mu_{2},\ldots,\mu_{n}$ be the eigenvalues of the skew
Laplacian
  matrix $\widetilde{SL}(G)=\widetilde{D}(G)-S(G)$, where $\widetilde{D}=diag(d^+_{1}-d^-_{1},d^+_{2}-d^-_{2},\ldots,d^+_{n}-d^-_{n})$
 and $S(G)=[s_{ij}]$ is the skew-adjacency matrix of $G$. Then from Observation \ref{obser1}, we have

$$\sum\limits_{i=1}^{n}\mu_{i}=\sum\limits_{i=1}^{n}(d^+_{i}-d^-_{i})=0. \eqno(1)$$

\noindent Note that $\sum\limits_{i<j}\mu_{i}\mu_{j}$ is equal to
the sum of the determinants of all $2\times2$ principle submatrices
of $\widetilde{SL}(G)$, which implies
\begin{eqnarray*}
   \sum\limits_{i<j}\mu_{i}\mu_{j} &=& \sum\limits_{i<j} det\left(
                                                              \begin{array}{cc}
                                                                d^+_{i}-d^-_{i} & -s_{ij} \\
                                                                -s_{ji} &  d^+_{j}-d^-_{j} \\
                                                              \end{array}
                                                            \right)
    \\
    &=& \sum\limits_{i<j}[(d^+_{i}-d^-_{i})(d^+_{j}-d^-_{j})-s_{ij}s_{ji}] \\
    &=& \sum\limits_{i<j}[(d^+_{i}-d^-_{i})(d^+_{j}-d^-_{j})+s_{ij}^{2}] \\
    &=& \sum\limits_{i<j}[(d^+_{i}-d^-_{i})(d^+_{j}-d^-_{j})]+m.
 \end{eqnarray*}
So $$~~\sum\limits_{i\neq j}\mu_{i}\mu_{j}=2\sum\limits_{i<j}\mu_{i}\mu_{j}
 =\sum\limits_{i\neq j}[(d^+_{i}-d^-_{i})(d^+_{j}-d^-_{j})]+2m. \eqno(2)$$

\noindent Combing $(1)$ and $(2)$, we get

\begin{eqnarray*}
 \sum\limits_{i=1}^{n}\mu_{i}^{2} &=& (\sum\limits_{i=1}^{n}\mu_{i})^{2}-\sum\limits_{i\neq j}\mu_{i}\mu_{j} \\
    &=& [\sum\limits_{i=1}^{n}(d^+_{i}-d^-_{i})]^{2}-[\sum\limits_{i\neq j}(d^+_{i}-d^-_{i})(d^+_{j}-d^-_{j})+2m] \\
    &=& \sum\limits_{i=1}^{n}(d^+_{i}-d^-_{i})^{2}-2m \\
    &=& 2M. ~~~~~~~~~~~~~~~~~~~~~~~~~~~~~~~~~~~~~~~~~~~~~~~~~~~~~~~~~~~~~~~~~~~~~~~~(3)
\end{eqnarray*}
Let $\widetilde{SL}(G)=[\ell_{ij}]$. By Schur's unitary triangularization
theorem \cite{Horn}, there exists a unitary matrix $U$ such that
$U^{*}\widetilde{SL}(G)U=T$, where $T=[t_{ij}]$ is an upper triangular matrix
with diagonal entries $t_{ii}=\mu_{i}$, $i=1,2,\ldots,n$. Therefore

$$\sum\limits_{i,j=1}^{n}|\ell_{ij}|^2=\sum\limits_{i,j=1}^{n}|t_{ij}|^2\geq
\sum\limits_{i=1}^{n}|t_{ii}|^2=\sum\limits_{i=1}^{n}|\mu_{i}|^2,
\eqno(4)$$

\noindent that is,
    $$\sum\limits_{i=1}^{n}|\mu_{i}|^2 \leq \sum\limits_{i,j=1}^{n}|\ell_{ij}|^2=
    \sum\limits_{i=1}^{n}(d^+_{i}-d^-_{i})^{2}+2m=2M_{1}. $$
Without loss of generality, assume that $|\mu_{1}|\geq
|\mu_{2}|\geq\ldots\geq|\mu_{n}|$. From Observation \ref{obser2}, we
know that $\mu_{n-i}=0$ for $i=0,1,\ldots,p-1$. Applying
Cauchy-Schwarz Inequality, it yields that
$$SLE(G)=\sum\limits_{i=1}^{n}|\mu_{i}|=\sum\limits_{i=1}^{n-p}|\mu_{i}|\leq \sqrt{(n-p)
\sum\limits_{i=1}^{n-p}|\mu_{i}|^2}=\sqrt{(n-p)\sum\limits_{i=1}^{n}|\mu_{i}|^2}
\leq \sqrt{2M_{1}(n-p)}.\eqno(5)$$

Now we turn to the proof of the left-hand inequality.

Since $\sum\limits_{i=1}^{n}\mu_{i}=0$, $\sum\limits_{i=1}^{n}\mu_{i}^{2}+2\sum\limits_{i<j}\mu_{i}\mu_{j}=0$.
Using $(3)$, we get $2\sum\limits_{i<j}\mu_{i}\mu_{j}=-2M$, which follows that
    $$2|M|=2|\sum\limits_{i<j}\mu_{i}\mu_{j}|\leq 2\sum\limits_{i<j}|\mu_{i}||\mu_{j}|. \eqno(6)$$
Using $(3)$ again,
    $$2|M|=|\sum\limits_{i=1}^{n}\mu_{i}^{2}|\leq \sum\limits_{i=1}^{n}|\mu_{i}|^{2}. \eqno(7)$$
From $(6)$ and $(7)$, we arrive at
\begin{center}
    $SLE(G)^{2}=(\sum\limits_{i=1}^{n}|\mu_{i}|)^{2}= \sum\limits_{i=1}^{n}|\mu_{i}|^{2}
    +2\sum\limits_{i<j}|\mu_{i}||\mu_{j}|\geq 4|M|$.
\end{center}
Consequently, $SLE(G)\geq 2\sqrt{|M|}$.

We proceed with the discussion for the sharpness of these bounds.

\begin{claim}
$SLE(G)= 2\sqrt{|M|}$ holds if and only if for each pair of
$\mu_{i_{1}}\mu_{j_{1}}$ and $\mu_{i_{2}}\mu_{j_{2}}$ ($i_{1}\neq
j_{1}$, $i_{2}\neq j_{2}$), there exists a non-negative real number
$k$ such that $\mu_{i_{1}}\mu_{j_{1}}=k\mu_{i_{2}}\mu_{j_{2}}$; and
for each pair of $\mu_{i_{1}}^{2}$ and $\mu_{i_{2}}^{2}$, there
exists a non-negative real number $\ell$ such that
$\mu_{i_{1}}^{2}=\ell\mu_{i_{2}}^{2}$.
\end{claim}

It follows from $(6)$ and $(7)$ that the equality is attained if and
only if $|\sum\limits_{i<j}\mu_{i}\mu_{j}|=
\sum\limits_{i<j}|\mu_{i}||\mu_{j}|$ and
$|\sum\limits_{i=1}^{n}\mu_{i}^{2}|=
\sum\limits_{i=1}^{n}|\mu_{i}|^{2}$. In other words, the equality
holds if and only if for each pair of $\mu_{i_{1}}\mu_{j_{1}}$ and
$\mu_{i_{2}}\mu_{j_{2}}$ ($i_{1}\neq j_{1}$, $i_{2}\neq j_{2}$),
there exists a non-negative real number $k$ such that
$\mu_{i_{1}}\mu_{j_{1}}=k\mu_{i_{2}}\mu_{j_{2}}$; and for each pair
of $\mu_{i_{1}}^{2}$ and $\mu_{i_{2}}^{2}$, there exists a
non-negative real number $\ell$ such that
$\mu_{i_{1}}^{2}=\ell\mu_{i_{2}}^{2}$, which proves Claim $1$.

A question arises: do such graphs exist ? The answer is yes. Let
$G_{1}$ be an orientation of $K_{2n,2n}$. Assume that $\{X, Y\}$ is
the bipartition of $G_{1}$. We divide $X\ (Y)$ into two disjoint
sets $X_{1}, X_{2}\ (Y_{1}, Y_{2})$ such that
$|X_{1}|=|X_{2}|=|Y_{1}|=|Y_{2}|=n$. The arc set is $\{(u_{1},
v_{1})|u_{1} \in X_{1}, v_{1}\in Y_{1}\}\bigcup\{(u_{2},
v_{2})|u_{2} \in X_{2}, v_{2}\in Y_{2}\}\bigcup\{(v_{1},
u_{2})|u_{2} \in X_{2}, v_{1}\in Y_{1}\}\bigcup\{(v_{2},
u_{1})|u_{1} \in X_{1}, v_{2}\in Y_{2}\}$; see Figure 1.

\begin{figure}[h,t,b,p]
\begin{center}
\scalebox{0.9}[0.9]{\includegraphics{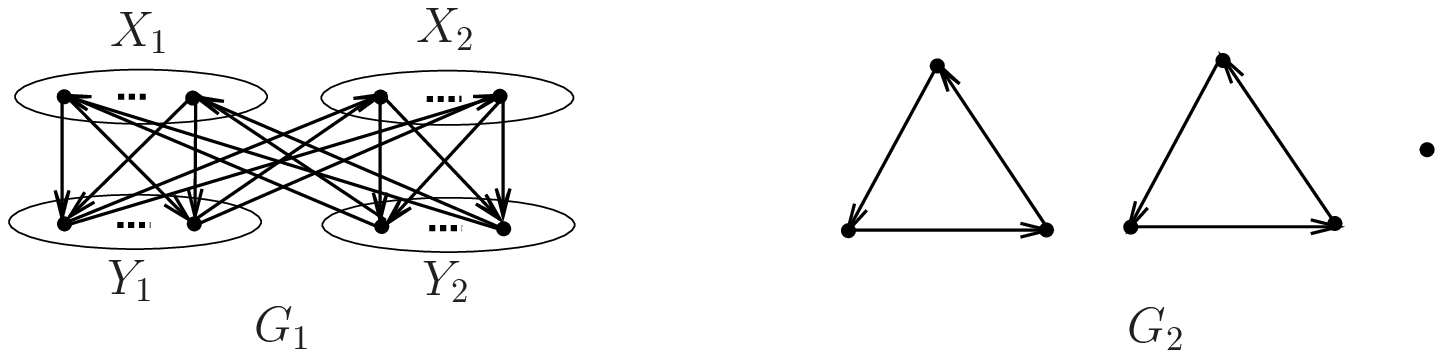}}\\[20pt]
 Figure 1 The graph for Claim 1.
\end{center}\label{fig1}
\end{figure}

Note that $d^+_{i}=d^-_{i}$ for each vertex $v_{i}$ in $G_{1}$. So
we get $2\sqrt{|M|}=2\sqrt{m}=4n$, and the skew Laplacian matrix of
$G_{1}$ is
\begin{center}
$\widetilde{SL}(G_{1})=-S(G_{1})=\left(
                          \begin{array}{cccc}
                            0 & 0 & -J & J \\
                            0 & 0 & J & -J \\
                            J & -J & 0 & 0 \\
                            -J & J & 0 & 0 \\
                          \end{array}
                        \right),$
\end{center}

\noindent where $J$ is the  $n\times n$ matrix in which each entry is $1$.

Then the skew Laplacian characteristic polynomial  $P_{\widetilde{SL}}(G_{1};
x)=det(xI-\widetilde{SL}(G_{1}))=x^{4n-2}(x^{2}+4n^{2})$, and the eigenvalues of
$\widetilde{SL}(G_{1})$ are $2ni$, $-2ni$, 0 with multiplicity 1, 1, $4n-2$,
respectively. Hence the skew Laplacian energy of $G_{1}$ is $4n$,
which implies that the lower bound is sharp.

\begin{claim}
 $SLE(G)= \sqrt{2M_{1}(n-p)}$ holds if and only if (i) $G$ is $0$-regular or
 (ii) for each $v_{i}\in V(G)$, $d^+_{i}=d^-_{i}$, and the eigenvalues of
 $\widetilde{SL}(G)$ are $0,ai,-ai$ $(a> 0)$ with multiplicity $p,\frac{n-p}{2},\frac{n-p}{2}$,
 respectively.
\end{claim}

It is evident from $(4)$ and $(5)$ that the equality holds if and
only if $T=[t_{ij}]$ is a diagonal matrix and
$|\mu_{1}|=|\mu_{2}|=\cdots=|\mu_{n-p}|$.

From Schur's unitary triangularzation theorem \cite{Horn}, we know
that $T=[t_{ij}]$ is a diagonal matrix if and only if $\widetilde{SL}(G)$ is a
normal matrix. That is
$$\widetilde{SL}^{*}(G)\cdot\widetilde{SL}(G)=\widetilde{SL}(G)\cdot\widetilde{SL}^{*}(G).$$
Since $\widetilde{SL}(G)=\widetilde{D}(G)-S(G)$, $\widetilde{SL}^{*}(G)=\widetilde{D}(G)+S(G)$,
we have
$$(\widetilde{D}(G)+S(G))\cdot(\widetilde{D}(G)-S(G))=(\widetilde{D}(G)-S(G))\cdot(\widetilde{D}(G)+S(G)).$$
By direct calculation, $S(G)\cdot\widetilde{D}(G)=\widetilde{D}(G)\cdot S(G)$.
Comparing the element on the $i$th row and the $j$th column of the
matrices on both sides, we arrive at
$$s_{ij}(d^+_{j}-d^-_{j})=(d^+_{i}-d^-_{i})s_{ij}.$$
If $v_{i}$ and $v_{j}$ are not adjacent (i.e., $s_{ij}=0$ ), then it
holds surely; if $v_{i}$ and $v_{j}$ are  adjacent, then
$s_{ij}\neq0$, and consequently $d^+_{i}-d^-_{i}=d^+_{j}-d^-_{j}$.

Now we are concerned with each component $C_{k}$ $(1\leq k\leq p)$
of $G$. Since $C_{k}$ is connected, any two vertices $u$, $w$ in
$C_{k}$ are connected by a path $P: u=v_{0}, v_{1},\cdots,v_{t}=w$.
Then $d^+(v_{i})-d^-(v_{i})=d^+(v_{i+1})-d^-(v_{i+1})$ for $0\leq
i\leq t-1$, which implies that $d^+(u)-d^-(u)=d^+(w)-d^-(w)$. It is
easy to see that $\sum\limits_{v\in V(C_{k})}[ d^+(v)-d^-(v)]=0$.
Therefore $d^+(v)-d^-(v)=0$ for each vertex $v$ in $C_{k}$. It
follows that $d^+_{i}=d^-_{i}$ for $v_{i}\in V(G)$, i.e.,
$\widetilde{D}(G)=0, \widetilde{SL}(G)=-S(G), \widetilde{SL}(C_{k})=-S(C_{k})$.

From Observation $\ref{obser2}$, we know that $0$ is an eigenvalue
of $\widetilde{SL}(G)$ with multiplicity at least $p$, and $0$ is also an
eigenvalue of $\widetilde{SL}(C_{k})(k=1,2\cdots p)$ with multiplicity at least
$1$. We distinguish the following two cases.

$Case \ 1:|\mu_{1}|=|\mu_{2}|=\cdots=|\mu_{n-p}|= 0$.

0 is the unique eigenvalue of $\widetilde{SL}(G)$ with multiplicity $n$, and
consequently $G$ is a 0-regular graph.

$Case \ 2:|\mu_{1}|=|\mu_{2}|=\cdots=|\mu_{n-p}|=a>0$.

That is, $0$ is an eigenvalue of $\widetilde{SL}(G)$ with multiplicity exactly
$p$, and the norms of all the other $n-p$ eigenvalues of $\widetilde{SL}(G)$ are
equal to $a$. Since $\widetilde{SL}(G)=-S(G)$ is a skew-symmetric matrix, its
eigenvalues are $0$ or pure imaginary numbers which appear in pairs.
We conclude that the eigenvalues of $\widetilde{SL}(G)$ are $0,ai,-ai$
$(a\geq0)$ with multiplicity $p,\frac{n-p}{2},\frac{n-p}{2}$,
respectively. This proves Claim 2.

Next, we give an example to verify the existence of such graphs. Let
$G_{2}=\alpha K_{3}\bigcup \beta K_{1}$, where $\alpha,\beta\in
\mathcal{N}$ and $3\alpha+\beta=n$; see Figure 1 for $\alpha=2,
\beta=1$. $K_{3}$ is oriented with the arc set $\{ (1,2), (2,3),
(3,1) \}$. The eigenvalues of $\widetilde{SL}(G_{2})$ are $\sqrt{3}i$,
$-\sqrt{3}i$, 0 with multiplicity $\alpha, \alpha, \alpha+\beta$,
respectively. Hence $SLE(G_{2})=2\sqrt{3}\alpha$ and
$\sqrt{2M_{1}[n-(\alpha+\beta)]}=\sqrt{4m\alpha}=2\sqrt{3}\alpha$,
which implies that the upper bound is also sharp.

Combining all above, we complete our proof.
\end{proof}

\begin{corollary}
Let $G$ be a simple digraph possessing $p$ components
$C_{1},C_{2},\ldots,C_{p}$. If $SLE(G)=\sqrt{2M_{1}(n-p)}$, then
each component $C_{i}$ is Eulerian with odd number of vertices.
\end{corollary}
\begin{proof}
If $G$ is a $0$-regular graph, each component of $G$ is an isolated
vertex, which obviously satisfies the conclusion. Otherwise, from
$Claim \  2$ we know that $d^+_{i}=d^-_{i}$ for each $v_{i}\in
V(C_{k})$, and hence $C_{k}$ is Eulerian. Furthermore, the eigenvalues
of $\widetilde{SL}(G)$ are $0,ai,-ai$ $(a> 0)$ with multiplicity
$p,\frac{n-p}{2},\frac{n-p}{2}$, respectively. It turns out that,
for each component $C_{k}$, $0$ is an eigenvalue of $\widetilde{SL}(C_{k})$ with
multiplicity exactly one and all the other eigenvalues are $ai,-ai$,
which appear in pairs. It follows that the number of vertices in
$C_{k}$ is odd.
\end{proof}

\begin{corollary}\label{cor2}

$SLE(G)\leq \sqrt{2M_{1}n}$.
\end{corollary}

\begin{corollary}\label{cor3}
If $G$ has no isolated vertices, then $SLE(G)\leq 2M_{1}$.
\end{corollary}
\begin{proof}
 If $G$ has no isolated vertices, then $n\leq 2m$. Therefore,

~~~~~~~~~~~~~~~~~~~~~~~~$ SLE(G)\leq \sqrt{2M_{1}n}\leq
2\sqrt{M_{1}m} \leq 2M_{1}.$
\end{proof}

We may mention that the bounds in Theorem \ref{thm5}, Corollary \ref{cor2} and
Corollary \ref{cor3} are in correspondence with those in
Theorem \ref{thm1} and Theorem \ref{thm4}.

\section{Concluding remarks}

Graph energy is one of the most active topics in chemical graph
theory. There have appeared several different definitions for the
(skew) Laplacian energy of undirected graphs and directed graphs.
Here we would like to summarize them below:

\emph{$1.$ The Laplacian energy of undirected graphs}

For an undirected graph $G$, there are two kinds of Lapalcian
energies $LE_{k}(G)=\sum\limits_{i=1}^{n}\mu^{2}$ and
$LE_{g}(G)=\sum\limits_{i=1}^{n}|\mu_{i}-\frac{2m}{n}|$, where
$\mu_{1},\mu_{2},\ldots,\mu_{n}$ are the eigenvalues of
$L(G)=D(G)-A(G)$.

\emph{$2.$ The Laplacian energy of directed graphs}

For a directed graph $G$, there is one kind of Laplacian energy
$LE_{m}(G)=\sum\limits_{i=1}^{n}\mu^{2}$, where
$\mu_{1},\mu_{2},\ldots,\mu_{n}$ are the eigenvalues of
$L^+(G)=D^+(G)-A^+(G)$.

\emph{$3.$ The skew Laplacian energy of simple directed graphs}

For a simple directed graph $G$, there are two kinds of skew
Laplacian energies $SLE_{k}(G)=\sum\limits_{i=1}^{n}\mu^{2}$ and
$SLE_{g}(G)=\sum\limits_{i=1}^{n}|\mu_{i}-\frac{2m}{n}|$, where
$\mu_{1},\mu_{2},\ldots,\mu_{n}$ are the eigenvalues of
$SL(G)=D(G)-S(G)$.

$4.$  In this paper, we introduce a new kind of skew Laplacian
matrix
$\widetilde{SL}(G)=\widetilde{D}(G)-S(G)=(D^+(G)-D^-(G))-(A^+(G)-A^-(G))$,
which is inspired by the popularly used Laplacian matrix
$L(G)=D(G)-A(G)=(D^+(G)+D^-(G))-(A^+(G)+A^-(G))$. From the
definition, we can see that the matrix $\widetilde{SL}(G)$ fully
reflects both the in-adjacency and the out-adjacency of a digraph
$G$. Moreover, it has some good properties such as
$trace(\widetilde{SL}(G))=0$, the sum of each row is $0$, $0$ is an
eigenvalue and $(1,,1,\ldots, 1)^T$ is an eigenvector, and so on.
These properties make $\widetilde{SL}(G))$ to be regarded as a skew
Laplacian matrix more reasonable. From the new skew Laplacian
matrix, we define a new skew Laplacian energy as
$SLE(G)=\sum\limits_{i=1}^{n}|\mu_{i}|$, where
$\mu_{1},\mu_{2},\ldots,\mu_{n}$ are the eigenvalues of
$\widetilde{SL}(G)=\widetilde{D}(G)-S(G)$. That new skew Laplacian
energy is well defined can also be seen from the following bounds,
which should be compared with those in Theorem \ref{thm1}:

 (i)~~$2\sqrt{|M|}\leq SLE(G)\leq \sqrt{2M_{1}n}$;

 (ii)~~$SLE(G)\leq \sqrt{2M_{1}(n-p)}$;

 (iii)~~If $G$ has no isolated vertices, then $SLE(G)\leq 2M_{1}$,

\noindent where
$M=-m+\frac{1}{2}\sum\limits_{i=1}^{n}(d^+_{i}-d^-_{i})^{2}$ and
$M_{1}=M+2m=m+\frac{1}{2}\sum\limits_{i=1}^{n}(d^+_{i}-d^-_{i})^{2}$.

\end{document}